\documentclass[12pt]{amsart}
\usepackage{amsmath,amssymb,amsthm}
\usepackage{verbatim}
\usepackage{color,xcolor}
\usepackage{enumerate}
\usepackage{mathrsfs}
\usepackage{mathtools}

\usepackage{bm}

\usepackage[colorlinks=true, bookmarks=true,pdfstartview=FitV, linkcolor=blue, citecolor=blue, urlcolor=blue]{hyperref}

\calclayout

\allowdisplaybreaks
\numberwithin{equation}{section}

\newtheorem{theorem}{Theorem}[section]
\newtheorem{lemma}[theorem]{Lemma}

\theoremstyle{definition}

\newtheorem{remark}[theorem]{Remark}

\theoremstyle{remark}

\newcommand{\R}{\mathbb R}

\newcommand{\B}{{\mathcal B}}
\newcommand{\F}{{\mathcal F}}

\def\Xint#1{\mathchoice 
{\XXint\displaystyle\textstyle{#1}}% 
{\XXint\textstyle\scriptstyle{#1}}% 
{\XXint\scriptstyle\scriptscriptstyle{#1}}% 
{\XXint\scriptscriptstyle\scriptscriptstyle{#1}}% 
\!\int} 
\def\XXint#1#2#3{{\setbox0=\hbox{$#1{#2#3}{\int}$} 
\vcenter{\hbox{$#2#3$}}\kern-.5\wd0}}

\def\avgint{\Xint-}

\DeclareMathOperator*{\esssup}{ess\,sup}
\DeclareMathOperator*{\essinf}{ess\,inf}

\newcommand{\pp}{{p(\cdot)}}

\newcommand{\Lp}{L^{p(\cdot)}}

\newcommand{\Pp}{\mathcal P}
\newcommand{\qq}{{q(\cdot)}}

\newcommand{\M}{\mathcal{M}}

\newcommand{\Rdf}{\mathcal{R}}

\newcommand{\qk}{{q_i(\cdot)}}
\newcommand{\tk}{{\theta_i(\cdot)}}

\begin{document}

\title[Multilinear fractional operators on weighted Hardy spaces]{Multilinear fractional Calder\'on-Zygmund  operators on weighted Hardy spaces}

\author[Cruz-Uribe]{David Cruz-Uribe, OFS}
\address{Department of Mathematics, University of Alabama, Tuscaloosa, AL 35487}
\email{dcruzuribe@ua.edu}

\author[Moen]{Kabe Moen}
\address{Department of Mathematics, University of Alabama, Tuscaloosa, AL 35487}
\email{kabe.moen@ua.edu}

\author[Nguyen]{Hanh Van Nguyen}
\address{Department of Mathematics, University of Alabama, Tuscaloosa, AL 35487}
\email{hvnguyen@ua.edu}

\subjclass[2010]{42B15, 42B20, 42B30, 42B35}

\keywords{Muckenhoupt weights, weighted Hardy spaces, variable Hardy
  spaces, multilinear fractional operators, Rubio de Francia extrapolation}

\thanks{The first author is supported by research funds from the Dean
  of the College of Arts \& Sciences, the University of Alabama. The
  second author is supported by the Simons Foundation.}

\date{February 28, 2019}

\begin{abstract}
  We prove norm estimates for multilinear fractional integrals acting
  on weighted and variable Hardy spaces.  In the weighted case we
  develop ideas we used for multilinear singular
  integrals~\cite{CMN18}.  For the variable exponent case, a key
  element of our proof is a new multilinear, off-diagonal version of
  the Rubio de Francia extrapolation theorem.
\end{abstract}

\maketitle

\section{Introduction}

The purpose of this paper is to continue the study of multilinear
operators on Hardy spaces begun in~\cite{CMN18,DCU-HN}.  In those
papers we considered multilinear Calder\'on-Zygmund operators and
multipliers.  Here we consider the multilinear fractional
Calder\'on-Zygmund operators introduced by Lin and
Lu~\cite{MR2353641}.  In the linear case, fractional
Calder\'on-Zygmund operators have been studied by a number of authors.  See, for
instance, the recent papers~\cite{MR3431617,MR3470665}.

Given positive integers $m,n$ and a real
number $0<\gamma<mn$, let
$K_\gamma$ be a function defined in $\mathbb{R}^{(m+1)n}$
away from the diagonal $x=y_1=\cdots=y_m$ that satisfies
the size condition
\begin{equation}
\label{eq-001}
|K_\gamma(x,y_1,\ldots,y_m)| \lesssim \Big(|x-y_1|+\cdots+|x-y_m|\Big)^{\gamma-mn},
\end{equation}
and the smoothness condition
\begin{equation}
\label{eq-002}
\sum_{i=1}^m\sum_{|\beta|=N}|\partial^{\beta}_{i}K_\gamma(x,y_1,\ldots,y_m)| \lesssim \Big(|x-y_1|+\cdots+|x-y_m|\Big)^{\gamma-mn-N}
\end{equation}
for some  sufficiently large integer $N$. We define the multilinear
fractional Calder\'on-Zygmund operator $T_\gamma$ by
\begin{equation*}
%\label{eq-003}
T_{\gamma}(f_1,\ldots,f_m)(x) = \int_{\mathbb{R}^{mn}}K_{\gamma}(x,y_1,\ldots,y_m)
f_1(y_1)\cdots f_m(y_m)\; d\vec{y}.
\end{equation*}

The simplest example of such an operator is the multilinear fractional
integral introduced by Kenig and Stein~\cite{MR1682725}:
\[ I_\gamma(f_1,\ldots,f_m)(x)
  = \int_{\R^{mn}} \frac{f_1(y_1)\cdots f_m(y_m)}
  {(|x-y_1|+\cdots+|x-y_m|)^{mn-\gamma}}\,d\vec{y}.  \]
They proved that for $1<p_1,\ldots,p_m\le
\infty$ and $q$ such that $\frac{1}{q} = \frac{1}{p_1} + \cdots + \frac1{p_m}-\frac{\gamma}{n}>0$,
\[ \|I_\gamma(f_1,\ldots,f_m)\|_{L^q}
\lesssim
C\|f_1\|_{L^{p_1}}\cdots \|f_m\|_{L^{p_m}}.
\]
Moreover, if $p_i=1$ for some $i$, then the above inequality is replaced by
the corresponding weak-type estimate.

Lin and Lu~\cite{MR2353641} proved Hardy space estimates for
multilinear fractional Calder\'on-Zygmund operators, generalizing the
results of Grafakos and Kalton~\cite{MR1852036} for multilinear
singular integrals and the results in the linear case for fractional
integrals due to Str\"omberg and Wheeden~\cite{MR766221} and Gatto,
{\em et al.}~\cite{MR784004}.    More precisely, they proved that if
$0<p_1,\ldots,p_m,q\le 1$, then
\[
\|T_\gamma(f_1,\ldots,f_m)\|_{L^q}
\lesssim
C\|f_1\|_{H^{p_1}}\cdots \|f_m\|_{H^{p_m}}.
\]
However, they had to make the restrictive assumption that
$0<\gamma<n$. 

Our first theorem is a generalization of the result of Lin and Lu to
weighted Hardy spaces.  To state it, we first recall some basic
definitions from the theory of  Muckenhoupt weights.   By a weight we mean a
non-negative, locally integrable function.  Given a weight $w$ and
$1<p<\infty$, we say $w$ is in the Muckenhoupt class $A_p$, denoted by
$w\in A_p$, if for every cube $Q$,
\[ \avgint_Q w\,dx \bigg(\avgint_Q w^{1-p'}\,dx\bigg)^{p-1} \leq C <
  \infty. \]
The smallest such constant $C$ is denoted by $[w]_{A_p}$.  The $A_p$
classes are nested:  $A_p\subset A_q$ if $p<q$.  Hence we can define
$A_\infty$ as the union of all the $A_p$ classes, and  define
$r_w=\inf\{ p : w\in A_p \}$.  
For 
$s>1$, we say that $w$ satisfies a reverse H\"older inequality with
exponent $s$, denoted by $w\in RH_s$, if for every cube $Q$,
\[ \left( \avgint_Q w^s\,dx\right)^{\frac{1}{s}}
  \leq C\avgint_Q w\,dx. \]
The infimum of all the constants for which this is true is denoted by
$[w]_{RH_s}$.   A weight is in $A_p$ for some $p>1$ if and only if it
is in $RH_s$ for some $s>1$.

\begin{theorem}
\label{th-001}
Let $0<\gamma<mn$.  Given $0<p_1,\ldots,p_m<\infty$, define $0<p<\infty$ by
\begin{equation}\label{eq-004}
\frac1p=\frac{1}{p_1}+\cdots+\frac{1}{p_m}>\frac{\gamma}{n},
\end{equation}
and define $0<q<\infty$  by
\begin{equation}
\label{eq-005}
 \frac1q = \frac{1}{p} - \frac{\gamma}{n}.
\end{equation}
Suppose that $(w_1,\ldots,w_m)$ is a vector of weights satisfying
$w_i\in RH_{\frac{q}{p}}$.  If $K_\gamma$ satisfies \eqref{eq-001} and
\eqref{eq-002} for some positive integer
$N> \max\{mn(\frac{r_{w_i}}{p_i}-1),1\le i\le m \}$, then
\begin{equation}
\label{eq-006}
\|T_\gamma(f_1,\ldots,f_m)\|_{L^q(\overline{w})} \lesssim
\prod_{i=1}^m\|f_i\|_{H^{p_i}(w_i)},
\end{equation}
where
\begin{equation*}
%\label{eq-0A070}
\overline{w} = \prod_{i=1}^mw_i^{\frac{q}{p_i}}.
\end{equation*}
\end{theorem}

\begin{remark}
  Even in the unweighted case Theorem~\ref{th-001} is a more general
  result than that of Lin and Lu, since we extend the values of
  $\gamma$ to the full range $0<\gamma<mn$.
\end{remark}

Below, we will prove Theorem~\ref{th-001} as a special case of a more
general result, Theorem \ref{thm:genver}.   It is more complicated to
state, since it requires
the existence of certain $q_i>p_i$ such that $w_i\in
RH_{\frac{q_i}{p_i}}$.  However, this result has the advantage that it
respects the product structure in the multilinear setting, in that we
do not have to assume an identical condition on each weight $w_i$.
This phenomenon does not appear in the diagonal case for multilinear
singular integrals considered in~\cite{CMN18}, but it does play a role
in the conditions for multilinear multipliers given in~\cite{DCU-HN}. 

\medskip

{As an application of our weighted estimates we extend our
results to the variable exponent setting}. Variable exponent spaces are
generalizations of the classical $L^p$ and $H^p$ spaces where the
constant exponent $p$ is replaced by an exponent function
$\pp$. Intuitively, $L^\pp$ consists of all functions such that
\[ \int_{\R^n} |f(x)|^{p(x)}\,dx < \infty.  \]
Harmonic analysis has been extensively studied on these spaces:
see~\cite{cruz-fiorenza-book} for the history and detailed
references.  The theory of variable exponent Hardy spaces $H^\pp$ was
introduced in~\cite{DCU-dw-P2014}. Our second main result is Theorem
\ref{th-002}.  For brevity, we defer some technical definitions to
Section~\ref{section:hardy}.

\begin{theorem}
\label{th-002}
Given $0<\gamma<mn$, let $p_i(\cdot)\in \Pp_0$, $1\leq i \leq m$, be
log-H\"older continuous exponent functions such that
$0<[p_i(\cdot)]_-\leq [p_i(\cdot)]_+ <\infty$ and
$$\frac{1}{[p_1(\cdot)]_+}+\cdots+  \frac{1}{[p_m(\cdot)]_+}>\frac{\gamma}{n}.$$
Define $\qq$ by 
$$
\frac{1}{q(\cdot)} = \frac{1}{p_1(\cdot)}+\cdots+\frac{1}{p_m(\cdot)} -\frac{\gamma}{n};
$$
then
\[
\|T_\gamma(f_1,\ldots,f_m)\|_{L^{q(\cdot)}}
\lesssim
 \prod_{i=1}^m \|f_i\|_{H^{p_i(\cdot)}}.
\]
\end{theorem}

The key tool in the proof Theorem~\ref{th-002} is a multilinear,
off-diagonal version of Rubio de Francia extrapolation,
Theorem~\ref{thm:xtpl}, which is of interest in its own right.  This
result generalizes earlier multilinear extrapolation theorems into the
scale of variable exponent spaces~\cite{CMN18,CN16} and also the
multilinear extrapolation theory in~\cite{CruzUribe:2018hg}.

\medskip

The remainder of this paper is organized as follows.  In
Section~\ref{section:prelim} we gather some definitions and
preliminary results needed in the proof of Theorem~\ref{th-001}.  In
Section~\ref{section:proof} we prove Theorem~\ref{th-001}.  Our proof
draws upon ideas from~\cite{CMN18,DCU-HN}, but significant
modifications were required to handle the fractional nature of the
kernel.  Finally, in Section~\ref{section:hardy} we give the necessary
definitions and prove Theorems~\ref{thm:xtpl} and~\ref{th-002}.

Throughout this paper, we will use $n$ to denote the dimension of the
underlying space, $\R^n$, and will use $m$ to denote the ``dimension''
of our multilinear operators.  By a cube $Q$ we will always mean a cube
whose sides are parallel to the coordinate axes, and for $\tau>1$ let
$\tau Q$ denote the cube with same center such that $\ell(\tau Q)=\tau
\ell(Q)$.  In particular, let $Q^*=2\sqrt{n}Q$ and $Q^{**}=(Q^*)^*$.  
By $C$, $c$, etc. we will mean
constants that may depend on the underlying parameters in the
problem.  The values of these
constants may change from line to line.  If we write $A\lesssim B$, we
mean that $A\leq cB$ for some constant~$c$.

\section{Preliminary results}
\label{section:prelim}

For $0\leq \gamma<n$, the fractional maximal function $M_\gamma$ is defined by
\begin{equation*}
%\label{eq-007}
M_\gamma f(x) = \sup_{Q}\ell(Q)^\gamma \Big( \avgint_Q|f(y)|dy \Big)\chi_{Q}(x).
\end{equation*}
When $\gamma=0$ we get the Hardy-Littlewood maximal operator and write
$Mf$ instead of $M_0f$.

For $1<p<\infty$ and $1<r<\infty$, given $w\in A_p$ we have the
Fefferman-Stein inequality:
\begin{equation}
\label{eqn:fef-stein}
\Big\|\Big(\sum_{k}M(f_k)^r\Big)^{\frac1r}\Big\|_{L^p(w)}
\lesssim
\Big\|\Big(\sum_{k}|f_k|^r\Big)^{\frac1r}\Big\|_{L^p(w)}
\end{equation}
A similar result holds for $M_\gamma$.
Given $0<\gamma<n,$ $1<p<\frac{n}{\gamma}$ and $\frac{1}{q} =
\frac1p-\frac{\gamma}{n}$, we say $w\in A_{p,q}$ if for all cubes $Q$,
\[ \bigg(\avgint_Q w^q\,dx\bigg)^{\frac{1}{q}}
  \bigg(\avgint_Q w^{-p'}\,dx\bigg)^{\frac{1}{p'}} \leq C <
  \infty. \]
Muckenhoupt and Wheeden \cite{MR0340523} showed that if  $w\in
A_{p,q}$,
then  
\begin{equation*}
%\label{eq-010}
\|M_\gamma f\|_{L^q(w^q)}
\lesssim
\|f\|_{L^p(w^p)}.
\end{equation*}
As a consequence of the off-diagonal Rubio de Francia extrapolation
\cite[Theorem~3.23]{MR2797562}, we have that for  $1<r<\infty$ and
$w\in A_{p,q}$,
\begin{equation}
\label{eq-011}
\Big\|\Big(\sum_{k}M_\gamma(f_k)^r\Big)^{\frac1r}\Big\|_{L^q(w^q)}
\lesssim
\Big\|\Big(\sum_{k}|f_k|^r\Big)^{\frac1r}\Big\|_{L^p(w^p)}.
\end{equation}

For $\gamma>0$, we have that
\begin{equation}
\label{eq-008}
\ell(Q)^{\gamma}\chi_{Q^*} \lesssim M_{\gamma\delta}(\chi_Q)^{\frac{1}{\delta}}
\end{equation}
for all $0<\delta\le 1$.  If we combine this estimate
with~\eqref{eq-011} we get the following vector-valued estimate.

\begin{lemma}
\label{lm-022}
Given $0<\gamma<\infty$ and $0<p<\frac{n}{\gamma}$, define $q>0$ by
$\frac1q = \frac1p-\frac{\gamma}n$.  Then for any $w\in RH_{\frac{q}{p}}$,
\[
\Big\|\sum_{j}\lambda_j\ell(Q_j)^{\gamma}\chi_{Q_j^*}\Big\|_{L^q(w^{\frac{q}{p}})}
\lesssim
\big\|\sum_{j}\lambda_j\chi_{Q_j}\big\|_{L^p(w)},
\]
where $\lambda_j>0$ and $\{Q_j\}_j$ is any sequence of cubes. 
\end{lemma}

\begin{remark} Lemma \ref{lm-022} was first proved in \cite{MR766221}
  when $1<p<n/\gamma$ in a two weight setting.  Our proof is much
  simpler.  For another proof that also uses extrapolation but
  avoids the vector-valued inequality see~\cite[Lemma 4.9]{CMN19}.
\end{remark}

\begin{proof}
  For each $\delta \in (0,p)$,  set $q_\delta = q/\delta$,
  $p_\delta = p/\delta$ and $u_\delta = w^{\frac1{p_\delta}}$. Since
  $w\in RH_{\frac{q}{p}}$, there exists $\delta>0$
sufficiently small so that
\[
u_\delta^{q_\delta}  = w^{\frac qp} \in A_{1+\frac{q_\delta}{(p_{\delta})'}}.
\]
Therefore, it follows from the definitions that
$u_\delta\in A_{p_\delta,q_\delta}$.   Then we can apply inequalities
\eqref{eq-008} and \eqref{eq-011} to get that
\begin{align*}
\Big\|\sum_{j}\lambda_j\ell(Q_j)^{\gamma}\chi_{Q_j^*}\Big\|_{L^q(w^{\frac{q}{p}})}
\lesssim&
\Big\|\sum_{j}\lambda_jM_{\gamma\delta}(\chi_{Q_j})^{\frac{1}{\delta}}
\Big\|_{L^q(w^{\frac{q}{p}})}\\
=&
\Big\|\Big(\sum_{j}\lambda_jM_{\gamma\delta}(\chi_{Q_j})^{\frac{1}{\delta}}\Big)^\delta
\Big\|_{L^{q_\delta}(w^{\frac{q_\delta}{p_\delta}})}^{\frac1{\delta}}\\
=&
\Big\|\Big(\sum_{j}\lambda_jM_{\gamma\delta}(\chi_{Q_j})^{\frac{1}{\delta}}\Big)^\delta
\Big\|_{L^{q_\delta}(u_\delta^{q_\delta})}^{\frac1{\delta}}\\
\lesssim&
\Big\|\Big(\sum_{j}\lambda_j\chi_{Q_j}\Big)^\delta
\Big\|_{L^{p_\delta}(u_\delta^{p_\delta})}^{\frac1{\delta}}\\
=&
\Big\|\sum_{j}\lambda_j\chi_{Q_j}
\Big\|_{L^{p}(w)}.
\end{align*}
\end{proof}

\begin{lemma}
\label{lm-023}
Let $\gamma, p$, and $q$ be real numbers as in Lemma \ref{lm-022} and
suppose $w\in RH_{\frac{q}{p}}\cap A_r$ for some $r>1$.  Then for
$\epsilon > \max(\frac{nr}{p},n)$ and any sequence $\{Q_j\}_j$ of
cubes,
\begin{equation*}
%\label{eq-012}
\Big\|
\sum_{j}\lambda_j
\frac{\ell(Q_j)^{\epsilon} \chi_{(Q_j^*)^c} }
{|\cdot - c_j|^{\epsilon-\gamma}}
\Big\|_{L^q(w^{\frac qp})}
\lesssim
\Big\|
\sum_j\lambda_j\chi_{Q_j}
\Big\|_{L^p(w)},
\end{equation*}
where $\lambda_j>0$ and $c_j$ is the center of the cube $Q_j$.
\end{lemma}

\begin{proof}
For each $j$, we first decompose $\mathbb{R}^n\setminus Q_{j}^*$ into
annuli and then into 
non-overlapping cubes $R_{j}^{kl}$ such that
\begin{equation*}
\label{eq-2011}
\mathbb{R}^n\setminus Q_{j}^* = 
\bigcup_{l=0}^{\infty}\bigcup_{k=1}^{3^n-1}R_{j}^{kl}
\end{equation*}
and such that
$|x-c_{j}|\approx 3^l\ell(Q_{j})\approx
\ell(R_{j}^{kl})$ for all $x\in R_{j}^{kl}$ and all
$1\le h\le 3^n-1$.  Consequently,  for any fixed 
$s>0$,
\begin{equation}
\label{eq-2012}
|x-c_{j}|^{-s}\chi_{(Q_{j}^*)^c}(x) \approx 
\sum_{l=0}^{\infty}
\sum_{k=1}^{3^n-1}
\big(3^l\ell(Q_{j})\big)^{-s}\chi_{R_{j}^{kl}}(x).
\end{equation}
Then by Lemma \ref{lm-022} and the equivalence \eqref{eq-2012} we have that
\begin{align*}
\Big\|
\sum_{j}\lambda_j
\frac{\ell(Q_j)^{\epsilon} \chi_{(Q_j^*)^c} }
{|\cdot - c_j|^{\epsilon-\gamma}}
\Big\|_{L^q(w^{\frac qp})}
\lesssim&
\Big\|
\sum_{j}
\sum_{l=0}^{\infty}
\sum_{k=1}^{3^n-1}
\lambda_j3^{-\epsilon l}
\ell(R_{j}^{kl})^{\gamma}
\chi_{R_{j}^{kl}}
\Big\|_{L^q(w^{\frac qp})}\\
\lesssim&
\Big\|
\sum_{j}
\sum_{l=0}^{\infty}
\sum_{k=1}^{3^n-1}
\lambda_j3^{-\epsilon l}
\chi_{R_{j}^{kl}}
\Big\|_{L^p(w)}\\
=&
\Big\|
\sum_{j}\lambda_j\ell(Q_j)^{\epsilon}
\sum_{l=0}^{\infty}
\sum_{k=1}^{3^n-1}
(3^{l}\ell(Q_j))^{-\epsilon}
\chi_{R_{j}^{kl}}
\Big\|_{L^p(w)}\\
\lesssim&
\Big\|
\sum_{j}\lambda_j\ell(Q_j)^{\epsilon}
|\cdot-c_{j}|^{-\epsilon}\chi_{(Q_{j}^*)^c}
\Big\|_{L^p(w)}\\
\lesssim&
\Big\|
\sum_{j}\lambda_j
M(\chi_{Q_j})^{\frac{\epsilon}{n}}
\Big\|_{L^p(w)}\\
\lesssim&
\Big\|\Big(
\sum_{j}\lambda_j
M(\chi_{Q_j})^{\frac{\epsilon}{n}}
\Big)^{\frac{n}{\epsilon}}
\Big\|_{L^{\frac{\epsilon p}{n}}(w)}^{\frac{\epsilon}n}\\
%\label{eq-2019}
\lesssim&
\Big\|
\sum_{j}\lambda_j
\chi_{Q_j}
\Big\|_{L^p(w)}.
\end{align*}
The last inequality holds because by our choice of $\epsilon$,
$\frac{\epsilon p}{n}>r$, so $w\in A_{\frac{\epsilon p}{n}}$ and we
can apply inequality~\eqref{eqn:fef-stein}.
\end{proof}

\section{The proof of Theorem~\ref{th-001}}
\label{section:proof}

In this section, we prove Theorem~\ref{th-001}, and 
as we said in the introduction, we  actually prove a more general result.  

\begin{theorem}
\label{thm:genver}
Given $0<\gamma<mn$ and $0<p_1,\ldots,p_m<\infty$ , define
$p$as in \eqref{eq-004} and $q$ as in \eqref{eq-005}.
Suppose that $q_i$ are such that $p_i<q_i<\infty$,
$1\le i\le m$, and $\frac1{q_1}+\cdots+\frac1{q_m} = \frac1{q}$, and 
$(w_1,\ldots,w_m)$ are weights  such that $w_i\in
RH_{\frac{q_i}{p_i}}$.  If
$K_\gamma$ satisfies \eqref{eq-001} and \eqref{eq-002} for some
positive integer $N> \max\{mn(\frac{r_{w_i}}{p_i}-1),1\le i\le m \}$,
then
\begin{equation*}
\|T_\gamma(f_1,\ldots,f_m)\|_{L^q(\overline{w})} \lesssim
\prod_{i=1}^m\|f_i\|_{H^{p_i}(w_i)},
\end{equation*}
where
\begin{equation*}
\overline{w} = \prod_{i=1}^mw_i^{\frac{q}{p_i}}.
\end{equation*}
\end{theorem}

\begin{remark}
  Theorem \ref{th-001}  follows from Theorem \ref{thm:genver} by
  taking $q_i=\frac{q}{p}p_i$. 
\end{remark}

\begin{proof}
Recall that for $w\in A_\infty$, $H^p(w)$ is defined as the set of all
distributions $f$ such that $\M_{N_0}f\in L^p(w)$; here $N_0$ is some
large constant whose precise value does not matter, though it will be
implicit in our constants.  For more information,
see~\cite{MR1011673}.  
Let $N$ be the positive integer as in the hypotheses of Theorem~\ref{th-001}. Define
\[
\mathcal{O}_N=
\big\{f\in \mathcal{C}_0^{\infty}\ :\
\int_{\mathbb{R}^n}x^{\beta}f(x)\,dx=0,\quad 0 \leq|\beta|\le N \big\}.
\]
Then $\mathcal{E}_i=\mathcal{O}_N\cap H^{p_i}(w_i)$ is dense in
$H^{p_i}(w_i)$ for all $1\le i\le m$ (see~\cite{CMN18,MR1011673}).   
As proved in \cite[Theorem~2.6]{CMN18} for each $f_i\in
\mathcal{E}_i$, $ 1\le i\le m$, 
we have a finite atomic decomposition:
\begin{equation}
\label{eq-201}
f_i = \sum_{k_i}\lambda_{i,k_i}a_{i,k_i},
\end{equation}
where $\lambda_{i,k_i}>0, |a_i|\le \chi_{Q_{i,k_i}}$ for some cube $Q_{i,k_i}$, $\int x^{\alpha}a_{i,k_i}dx=0$ for all $|\alpha|\le N$, and
\begin{equation}
\label{eq-202}
\big\|\sum_{k_i}\lambda_{i,k_i}\chi_{Q_{i,k_i}}\big\|_{L^{p_i}(w_i)}
\lesssim
\|f_i\|_{H^{p_i}(w_i)}.
\end{equation}

By a standard density argument, it will suffice to show that
inequality \eqref{eq-006} holds for $f_i$ of the form \eqref{eq-201}.
Define  $0<\gamma_i<\infty$ by 
\begin{equation*}
%\label{eq-204}
\frac{\gamma_i}{n} 
 = \frac 1{p_i} - \frac 1{q_i},\quad 1\le i\le m.
\end{equation*}
From the hypothesis \eqref{eq-005} we get
\begin{equation}
\label{eq-203}
\sum_{i=1}^m\gamma_i = \gamma,\qquad 0< \gamma_i < \frac n{p_i},\quad 1\le i\le m.
\end{equation}
By the multilinearity of $T_\gamma$ we get that
\begin{align*}
T_\gamma(f_1,\ldots,f_m)(x) =& \sum_{k_1,\ldots,k_m}
\lambda_{1,k_1}\cdots \lambda_{m,k_m}
T_\gamma(a_{1,k_1},\ldots,a_{m,k_m})(x).
\end{align*}
Given a cube $Q$ and $\vec{k} = (k_1,\ldots,k_m)$,  define
\begin{equation*}
%\label{eq-206}
E_{\vec{k}} = 
\cap_{i=1}^mQ_{i,k_i}^*,\qquad F_{\vec{k}} = \mathbb{R}^n\setminus E_{\vec{k}}.
\end{equation*}
Then we can decompose $T_\gamma(f_1,\ldots,f_m) = G_1+G_2$ where
\begin{align*}
G_1 = &
\sum_{k_1,\ldots,k_m}
\lambda_{1,k_1}\cdots\lambda_{m,k_m}
T_\gamma(a_{1,k_1},\ldots,a_{m,k_m})\chi_{E_{\vec{k}}},\\
G_2=&\sum_{k_1,\ldots,k_m}
\lambda_{1,k_1}\cdots\lambda_{m,k_m}
T_\gamma(a_{1,k_1},\ldots,a_{m,k_m})\chi_{F_{\vec{k}}}.
\end{align*}

To estimate the first term, we may assume that $E_{\vec{k}}$ is not
empty.  With $\gamma_i$ as defined by~\eqref{eq-203} we can estimate
$T_\gamma(a_{1,k_1},\ldots,a_{m,k_m})(x)$ for all $x\in E_{\vec{k}}$
as follows:
\begin{align}
|T_\gamma(a_{1,k_1},\ldots,a_{m,k_m})(x)|\le &
\int_{\mathbb{R}^{mn}}
\frac{
\chi_{Q_{1,k_1}}(y_1)\cdots \chi_{Q_{m,k_m}}(y_m)\;d\vec{y}
}
{
\big(|x-y_1|+\cdots+|x-y_m|\big)^{mn-\gamma}
}\nonumber\\
\le&
\prod_{i=1}^m
\Big(\int_{Q_{i,k_i}}|x-y_i|^{\gamma_i-n}dy_i\Big)\chi_{Q_{i,k_i}^*}(x)\nonumber\\
\lesssim&
\prod_{i=1}^m\ell(Q_{i,k_i})^{\gamma_i}\chi_{Q_{i,k_i}^*}(x).
\label{eq-0B0310}
\end{align}
We can now estimate the quasi-norm of $G_1$ as follows:
\begin{align}
\|G_1\|_{L^q(\overline{w})} = &
\Big\|\sum_{k_1,\ldots,k_m}
\lambda_{1,k_1}\cdots\lambda_{m,k_m}
T_\gamma(a_{1,k_1},\ldots,a_{m,k_m})\chi_{E_{\vec{k}}}
\Big\|_{L^q(\overline{w})}\nonumber\\
\lesssim&
\Big\|\sum_{k_1,\ldots,k_m}
\lambda_{1,k_1}\cdots\lambda_{m,k_m}
\prod_{i=1}^m\ell(Q_{i,k_i})^{\gamma_i}\chi_{Q_{i,k_i}^*}
\Big\|_{L^q(\overline{w})}\nonumber\\
=&
\Big\|
\prod_{i=1}^m
\Big(\sum_{k_i}\lambda_{i,k_i}\ell(Q_{i,k_i})^{\gamma_i}\chi_{Q_{i,k_i}^*}\Big)
\Big\|_{L^q(\overline{w})}\nonumber\\
\label{eq-315}
\le&
\prod_{i=1}^m
\Big\|
\sum_{k_i}\lambda_{i,k_i}\ell(Q_{i,k_i})^{\gamma_i}\chi_{Q_{i,k_i}^*}
\Big\|_{L^{q_i}(w_i^{q_i/p_i})}\\
\label{eq-316}
\lesssim&
\prod_{i=1}^m
\Big\|
\sum_{k_i}\lambda_{i,k_i}\chi_{Q_{i,k_i}}
\Big\|_{L^{p_i}(w_i)};
\end{align}
for inequality \eqref{eq-315} we used H\"older's inequality, and
for~\eqref{eq-316} we used Lemma \ref{lm-022}.

\medskip

To estimate $G_2$, fix $x\in F_{\vec{k}}$.  Then there exists a
non-empty subset $\Lambda$ of $\{1,\ldots,m\}$ such that
$x\notin Q_{i,k_i}^*,$ for all $i\in \Lambda$ and $x\in Q_{j,k_j}^*$
for all $1\le j\le m, j\not\in \Lambda$. Let $Q_{i_0,k_{i_0}}$, for
some $i_0\in \Lambda$, be the cube with smallest length among
$Q_{i,k_i},i\in\Lambda$ and let $c_{i,k_i}$ be the center of the cube
$Q_{i,k_{i}}$.  Note that since $x\notin Q_{i_0,k_{i_0}}^*$,
$|x-y_0| \lesssim |x-c_{i_0,k_{i_0}}|$ for all
$y_{i_0}\in Q_{i_0,k_{i_0}}$.  Let
\[
P_N(x,y_1,\ldots,,c_{i_0,k_{i_0}},\ldots,y_m)
=
\sum_{|\beta|<N}\frac{ \partial^{\beta}_{i_0} K_\gamma(x,y_1,\ldots,,c_{i_0,k_{i_0}},\ldots,y_m) }{\beta!}(y_{i_0}-c_{i_0,k_{i_0}})^\beta
\]
be the Taylor polynomial of $K_\gamma$.  
Then the  cancellation conditions satisfied by the atoms $a_{i_0,k_{i_0}}$ and the smoothness of $K_\gamma$ in \eqref{eq-002} imply that
\begin{align}
\lefteqn{|T_\gamma(a_{1,k_1},\ldots,a_{m,k_m})(x)|}\nonumber\\
\le&
\int |K_\gamma(x,y_1,\ldots,y_m) - P_N(x,y_1,\ldots,,c_{i_0,k_{i_0}},\ldots,y_m)|
\prod_{i=1}^m|a_{i,k_i}(y_i)|d\vec{y}\nonumber\\
\lesssim&
\int \frac{\ell(Q_{i_0,k_{i_0}})^N
\prod_{i=1}^m\chi_{Q_{i,k_i}}(y_i)\;d\vec{y} }{(|x-y_1|+\cdots+|x-y_m|)^{mn+N-\gamma}}\nonumber\\
\lesssim&
\prod_{i\in\Lambda}
\int \frac{\ell(Q_{i_0,k_{i_0}})^{\frac{N}{|\Lambda|}}\chi_{Q_{i,k_{i}}}(y_i)\;dy_i}{|x-y_i|^{n+\frac{N}{|\Lambda|}-\gamma_i}}
\cdot
\prod_{i\notin \Lambda}
\int\frac{ \chi_{Q_{i,k_i}(y_i)}dy_i } { |x-y_i|^{n-\gamma_i} }\nonumber\\
\lesssim&
\prod_{i\in\Lambda}
\int \frac{\ell(Q_{i,k_i})^{\frac{N}{|\Lambda|}}\chi_{Q_{i,k_{i}}}(y_i)\;dy_i}
{|x-y_i|^{n+\frac{N}{|\Lambda|}-\gamma_i}}
\cdot
\prod_{i\notin \Lambda}
\int\frac{ \chi_{Q_{i,k_i}(y_i)}dy_i } { |x-y_i|^{n-\gamma_i} }\nonumber\\
\lesssim&
\prod_{i\in\Lambda}
\frac{\ell(Q_{i,k_i})^{\epsilon_N}\chi_{(Q_{i,k_i}^*)^c}(x) }
{|x-c_{i,k_i}|^{\epsilon_N -\gamma_i}}
\cdot
\prod_{i\notin\Lambda}\ell(Q_{i,k_i})^{\gamma_i}\chi_{Q_{i,k_i}^*}(x),
\label{eq-0B0310}
\end{align}
where $\epsilon_N = n+\frac Nm$.  (For details of this calculation,
see~\cite[Lemma~3.6]{CMN18}.)  Since
$w_i\in RH_{\frac{q_i}{p_i}}\subset A_\infty$, $1\le i\le m$,
for all $r_i>r_{w_i}$,  $w_i\in A_{r_i}$.  By our assumption on $N$ in
the hypotheses, we can choose $r_i$ close enough to $r_{w_i}$ so that
\[
\epsilon_N = n+\frac Nm > \frac{r_i}{p_i}.
\]
If we combine the above estimates we get
\begin{align*}
\|G_2\|_{L^q(\overline{w})}
\le&
\Big\|
\sum_{k_1,\ldots,k_m}
\lambda_{1,k_1}\cdots\lambda_{m,k_m}
|T_\gamma(a_{1,k_1},\ldots,a_{m,k_m})|
\chi_{F_{\vec{k}}}
\Big\|_{L^q(\overline{w})}\\
\lesssim&
\Big\|
\sum_{k_1,\ldots,k_m}
\prod_{i\in\Lambda}
\frac{\lambda_{i,k_i}\ell(Q_{i,k_i})^{\epsilon_N}\chi_{(Q_{i,k_i}^*)^c}}
{|\cdot-c_{i,k_i}|^{\epsilon_N -\gamma_i}}
\cdot
\prod_{i\notin\Lambda}
\lambda_{i,k_i}\ell(Q_{i,k_i})^{\gamma_i}\chi_{Q_{i,k_i}^*}
          \Big\|_{L^q(\overline{w})}.
\end{align*}
By H\"older''s inequality, and Lemmas~\ref{lm-022} and~\ref{lm-023} we
get 
\begin{align}
\|G_2 \|_{L^q(\overline{w}^q)}
\lesssim&
\prod_{i\in\Lambda}
\Big\|
\sum_{k_i}\lambda_{i,k_i}
\frac{\ell(Q_{i,k_i})^{\epsilon_N}\chi_{(Q_{i,k_i}^*)^c}}
{|\cdot-c_{i,k_i}|^{\epsilon_N -\gamma_i}}
\Big\|_{L^{q_i}(w_i^{q_i/p_i})}\nonumber\\
\notag
&\times
\prod_{i\notin\Lambda}
\Big\|
\lambda_{i,k_i}\ell(Q_{i,k_i})^{\gamma_i}\chi_{Q_{i,k_i}^*}
\Big\|_{L^{q_i}(w^{q_i/p_i})}\\
\label{eq-3025}
\lesssim&
\prod_{i=1}^m
\Big\|
\sum_{k_i}\lambda_{i,k_i}\chi_{Q_{i,k_i}}
\Big\|_{L^{p_i}(w_i)}.
\end{align}
Combining \eqref{eq-316} and \eqref{eq-3025} we get
\[
\|T_\gamma(f_1,\ldots,f_m)\|_{L^{q}(\overline{w})}
\lesssim
\prod_{i=1}^m
\Big\|
\sum_{k_i}\lambda_{i,k_i}\chi_{Q_{i,k_i}}
\Big\|_{L^{p_i}(w_i)},
\]
which, when combined with \eqref{eq-202}, gives us the desired estimate
for $T_\gamma$. 
\end{proof}

\begin{remark}
  If $0<\gamma<(m-l)n$ for some $1\le l<m$, then we can allow at most
  $l$ exponents among the  $\{p_1,\ldots,p_m\}$ to be infinite and the
  conclusion of Theorem~\ref{thm:genver} is still true, replacing
  $H^{p_i}(w_i)$ with $L^\infty$.   To see this, first note that we may assume that
  $p_{m-l+1}=\cdots=p_m=\infty$. Then we can  integrate in
  $y_{m-l+1},\ldots,y_m$ to estimate \eqref{eq-0B0310} as follows:
\begin{align*}
& |T_\gamma(a_{1,k_1},\ldots,a_{m-l,k_{m-l}},f_{m-l+1},\ldots,f_m)(x)|\\
& \qquad \le 
\int_{\mathbb{R}^{mn}}
\frac{
\chi_{Q_{1,k_1}}(y_1)\cdots \chi_{Q_{m-l,k_{m-l}}}(y_{m-l})f_{m-l+1}(y_{m-l+1})\cdots f_m(y_m)\;d\vec{y}
}
{
\big(|x-y_1|+\cdots+|x-y_l|\big)^{nl-\gamma}
}\\
& \qquad \le 
\int_{\mathbb{R}^{nl}}
\frac{
\chi_{Q_{1,k_1}}(y_1)\cdots \chi_{Q_{m-l,k_{m-l}}}(y_{m-l})\;dy_1\cdots dy_l
}
{
\big(|x-y_1|+\cdots+|x-y_l|\big)^{nl-\gamma}
}
\prod_{i=l+1}^m\|f_i\|_{L^{\infty}}
\\
& \qquad \le
\prod_{i=1}^l
\Big(\int_{Q_{i,k_i}}|x-y_i|^{\gamma_i-n}dy_i\Big)\chi_{Q_{i,k_i}^*}(x)
\prod_{i=l+1}^m\|f_i\|_{L^{\infty}}
\\
& \qquad \lesssim
\prod_{i=1}^l\ell(Q_{i,k_i})^{\gamma_i}\chi_{Q_{i,k_i}^*}(x)
\prod_{i=l+1}^m\|f_i\|_{L^{\infty}}.
\end{align*}
If we now repeat the argument in the proof of Theorem \ref{th-001}, we get
\[
\|T_\gamma(f_1,\ldots,f_m)\|_{L^{q}(\overline{w})}
\lesssim
\prod_{i=1}^l
\|f_i\|_{H^{p_i}(w_i)}
\prod_{i=l+1}^m\|f_i\|_{L^{\infty}}.
\]

In their work, Lin and Lu~\cite[Theorem~2.1]{MR2353641} assumed an unweighted
estimate similar to this one when $l=m-1$.
This helps to explain the restriction $0<\gamma<n$
in their results.
\end{remark}

\section{Boundedness on Variable Hardy Spaces} 
\label{section:hardy}

In this section, we state and prove the analogue of 
Theorem~\ref{th-001} on the variable
exponent Hardy spaces, $H^\pp$.  
To do so, we first recall
some basic facts about the variable exponent Lebesgue and Hardy spaces.  For complete
background information we refer the reader to~\cite{cruz-fiorenza-book}.

Let $\Pp_0(\R^n)$ be the set of all measurable functions $\pp : \R^n
\rightarrow (0,\infty)$.  Define
\[ [p(\cdot)]_- = \essinf_{x\in \mathbb{R}^n} p(x), 
\qquad [p(\cdot)]_+ = \esssup_{x\in \mathbb{R}^n}
p(x). \]
Given $\pp\in \Pp_0(\R^n)$ define $\Lp=\Lp(\R^n)$ to be the set of all
measurable functions $f$ such that for some $\lambda>0$,
\[ \rho(f/\lambda) = \int_{\mathbb{R}^n}
\left(\frac{|f(x)|}{\lambda}\right)^{p(x)}\,dx< \infty.  \]
This becomes a quasi-Banach space with the quasi-norm
$$ \|f\|_{L^{p(\cdot)}} = \inf\left\{ \lambda > 0 : \rho(f/\lambda)\leq 1
\right\}.
$$
If $[p(\cdot)]_-\geq 1$, then $\|\cdot\|_{L^\pp}$ is a norm and  $\Lp$ is a
Banach space.  For all $p>0$, if $\pp=p$ a constant, then
$\Lp=L^p$ with equality of norms, so the variable exponent Lebesgue
spaces are a generalization of the classical $L^p$ spaces.

Let $\B$ denote the collection of exponents $\pp$ such that the
Hardy-Littlewood maximal operator is bounded on $\Lp$.  A sufficient
(but not necessary) condition for $\pp\in \B$ is that
$1<[p(\cdot)]_-\leq [p(\cdot)]_+<\infty$ and $\pp$ is log-H\"older
continuous: there exist constants $C_0$, $C_\infty$ and $p_\infty$
such that
\[ 
 |p(x)-p(y)| \leq \frac{C_0}{-\log(|x-y|)}, \qquad
  0<|x-y|<\frac{1}{2}, 
\]
and
\[
 |p(x) - p_\infty| \leq \frac{C_\infty}{\log(e+|x|)}. 
\]

Given $\pp\in \Pp_0(\R^n)$, the variable Hardy space $H^\pp$ is
defined to be the set of all distributions $f$ such that $\mathcal{M}_{N_0}f\in
\Lp$.  Again, we here assume $N_0>0$ is a sufficiently large integer
so that all the standard definitions of the classical Hardy spaces are
equivalent in $H^\pp$.  For further details, see~\cite{CMN18,DCU-dw-P2014}.  

We will prove norm inequalities in the variable exponent Lebesgue and Hardy spaces
from the corresponding weighted norm inequalities by
applying the theory of Rubio de Francia extrapolation.   In the linear
case this approach was introduced
in~\cite{cruz-fiorenza-book,cruz-uribe-fiorenza-martell-perez06,CMP11},
and multilinear versions of extrapolation into the variable Lebesgue
spaces were proved in~\cite{CMN18,CN16}.  Here we need a
generalization of these results similar to the multilinear extrapolation theorem
proved in~\cite{CruzUribe:2018hg}.    
We state our extrapolation results in terms of extrapolation $(m+1)$-tuples;
for more on this approach, see~\cite{CMP11,CN16}.

\begin{theorem}\label{thm:xtpl}
  Let $\F = \big\{ (f_1,\ldots,f_m,F)\big\}$ be a family of
  $(m+1)$-tuples of non-negative, measurable functions on $\R^n$.
  Given $0<\gamma < mn$ and exponents $0<p_i<\infty$,
  $1\leq i \leq m$, such that~\eqref{eq-004} holds,
  define $q>0$ by~\eqref{eq-005}. 
  Suppose  that for all exponents $p_i<q_i <\infty$ such that
  \[ \frac{1}{q} =  \frac{1}{q_1}+\cdots+\frac{1}{q_m}, \]
  and for all weights $w_i \in RH_{q_i/p_i}$, with
 $\overline{w} = \prod_{i=1}^m w_i^{q/p_i}$,
 we have that
\begin{equation} \label{eqn:xtpl1}
\|F\|_{L^{q}(\overline{w})} \lesssim  \prod_{i=1}^m\|f_i\|_{L^{p_i}(w_i)}
\end{equation}
for all $(f_1,\ldots,f_m,F)\in \F$ such that $F\in L^{q}(\overline{w})$,
 where the
implicit constant depends only on $n$, $p_i$, and
$[w_i]_{RH_{q_i/p_i}}$.

Let $p_i(\cdot)\in \Pp_0$, $1\leq i \leq m$,
be such that  each $p_i(\cdot)$ is log-H\"older continuous, $p_i< [p_i(\cdot)]_-\leq [p_i(\cdot)]_+<\infty$,
and
\begin{equation} \label{eqn:r-defn}
\frac{1}{[p_1(\cdot)]_+} + \cdots + \frac{1}{[p_m(\cdot)]_+} >
  \frac{\gamma}{n}. 
\end{equation}
Define $\qq$ by
\begin{equation} \label{eqn:s-defn}
\frac{1}{q(\cdot)} = \frac{1}{p_1(\cdot)}+\cdots+\frac{1}{p_m(\cdot)} -\frac{\gamma}{n}.
\end{equation}
Then for all $(f_1,\ldots,f_m,F)\in \F$ such that $\|F\|_{L^{q(\cdot)}}<\infty$,
\begin{equation} \label{eqn:xtpl2}
\|F\|_{L^{q(\cdot)}} \lesssim \prod_{i=1}^m\|f_i\|_{L^{p_i(\cdot)}}.
\end{equation}
The implicit constant only depends
on $n$, $[p_i(\cdot)]_-$, $[p_i(\cdot)]_+$, and the log-H\"older
constants of $p_i(\cdot)$.
\end{theorem}

\begin{remark}
It will be clear from the proof that we can weaken the hypothesis that each
$p_i(\cdot)$ is log-H\"older continuous, and instead assume that the maximal
operator is bounded on a certain family of variable exponent Lebesgue spaces. Details are left to the
interested reader.
\end{remark}

\begin{remark}
Theorem~\ref{thm:xtpl} is stated so that its hypotheses coincide with
the weighted results in Theorem~\ref{thm:genver}.  We can also prove
an extrapolation theorem starting from the weaker conclusion given in
Theorem~\ref{th-001}.  The proof below can be modified, but we 
need to assume that the exponents $p_i(\cdot)$ have bounded
oscillation:  more precisely, that
\[ p_i < [p_i(\cdot)]_- \leq [p_i(\cdot)]_+ < \frac{qp_i}{q-p}.  \]
Details of this result are left to the interested reader.  The fact
that we can remove this upper bound on $[p_i(\cdot)]_+$ is another
reason for proving the stronger result in  Theorem~\ref{thm:genver}. 
\end{remark}

\begin{proof}
For our proof we need a family of Rubio de Francia
iteration algorithms.  To construct them, we will define some exponent
functions and show that the maximal operator is bounded on the
associated variable Lebesgue space.  
By \eqref{eqn:r-defn}, for each $i$, $1\leq i \leq m$, we can fix $\gamma_i>0$ such that
\[ \gamma = \sum_{i=1}^n \gamma_i, \]
and $[p_i(\cdot)]_+<n/\gamma_i$.  Define $q_i>0$ by
\[ \frac{1}{p_i}-\frac{1}{q_i} = \frac{\gamma_i}{n} \]
and define the variable exponents $q_i(\cdot)$ by
\[ \frac{1}{p_i(\cdot)} - \frac{1}{q_i(\cdot)} = \frac{\gamma_i}{n}.  \]
But then by~\eqref{eqn:s-defn} we have that
\[ \frac{1}{q(\cdot)} =
  \frac{1}{q_1(\cdot)}+\cdots+\frac{1}{q_m(\cdot)}, \]
and so
\[ \frac{1}{[\qq]_-} \le \sum_{i=1}^m \frac{1}{[p_i(\cdot)]_-}  -
  \frac{\gamma}{n}
  < \sum_{i=1}^m \frac{1}{p_i}  -
  \frac{\gamma}{n} = \frac{1}{q}.  \]
Therefore, if we define $\overline{q}(\cdot)= \qq/q$, then
$[\overline{q}(\cdot)]_- >1$.  Similarly, if we define
$\overline{p}_i(\cdot)= p_i(\cdot)/p_i$,
then $[\overline{p}_i(\cdot)]_- >1$.

Now let $\sigma_i(\cdot)= \frac{p_i}{q_i} \overline{p}_i'(\cdot)$.  We
claim that $[\sigma_i(\cdot)]_->1$.  However, this inequality follows
from some standard estimates for dual exponents in the variable
exponent Lebesgue spaces \cite[p.~14]{cruz-fiorenza-book}: this
inequality is equivalent to $[\overline{p}_i'(\cdot)]_->
\frac{q_i}{p_i}$, which in turn is equivalent to $[\overline{p}_i(\cdot)]_+'>
\frac{q_i}{p_i}$, and this in turn is equivalent to
\[ [p_i(\cdot)]_+ < p_i\left(\frac{q_i}{p_i}\right)'  =
  \frac{n}{\gamma_i}, \]
which we know to hold.  

We also have that each $\sigma_i(\cdot)$ is log-H\"older continuous
since each $p_i(\cdot)$ is.   Therefore, the Hardy-Littlewood maximal
operator is bounded on $L^{\sigma_i(\cdot)}$.  Hence,
 we can define the iteration operator $\Rdf_i$, acting on
non-negative functions $h$, by
\[ \Rdf_i h(x) = \sum_{j=0}^\infty \frac{M^jh(x)}{2^j
    \|M\|_{L^{\sigma_i(\cdot)}}^j}, \]
where $M^jh = M\circ \cdots \circ Mh$ is $j$ iterates of the maximal
operator, and $M^0h = h$.  Then by a standard argument
\cite[p.~210]{cruz-fiorenza-book} and a rescaling property of $A_1\cap RH_s$ weights
\cite{cruz-uribe-neugebauer95},
we have the following:
\begin{enumerate}

\item $h(x) \leq \Rdf_i h(x)$;

\item $\|\Rdf_i h\|_{L^{\sigma_i(\cdot)}} \leq
  2\|h\|_{L^{\sigma_i(\cdot)}}$;

\item $\Rdf_i h \in A_1$, and $[\Rdf_i h]_{A_1} \leq 2
  \|M\|_{L^\sigma_i(\cdot)}$;

\item $(\Rdf_i h)^{p_i/q_i} \in A_1 \cap RH_{q_i/p_i}$, and
  $[(\Rdf_i h)^{p_i/q_i}]_{RH_{q_i/p_i}}$ depends only on
  $[\Rdf_i h]_{A_1}$. 

\end{enumerate}

Define a family of auxiliary exponents $\theta_i$, $1\leq i \leq
m$, by 
\[ \tk = \frac{q \overline{q}'(\cdot)}{p_i\overline{p}_i'(\cdot)}. \]
Then 
\begin{multline} \label{eqn:tk}
 \sum_{i=1}^m \tk
= q \overline{q}'(\cdot) \sum_{i=1}^m \frac{1}{p_i\overline{p}_i'(\cdot)}
= q \overline{q}'(\cdot) \sum_{i=1}^m \frac{1}{p_i}\bigg(1-\frac{p_i}{\overline{p}_i(\cdot)}\bigg) \\
= q \overline{q}'(\cdot)\sum_{i=1}^m \bigg(\frac{1}{p_i}-\frac{1}{p_i(\cdot)} \bigg)
= \overline{q}'(\cdot)- \frac{\overline{q}'(\cdot)}{\overline{q}(\cdot)} = 1. 
\end{multline}

\medskip

We can now prove the desired inequality.   Fix $(f_1,\ldots,f_m,F)\in \F$ such
that $F\in L^\qq$.  Since $\overline{q}(\cdot)>1$, by
rescaling and the associate norm in variable exponent Lebesgue
spaces~\cite[Prop.~2.18, Thm.~2.34]{cruz-fiorenza-book},  there exists $h\in L^{\overline{q}'(\cdot)}$,
$\|h\|_{L^{\overline{q}'(\cdot)}}=1$, such that
\begin{multline} \label{eqn:xtrpl3}
\|F\|_{L^{\qq}}^q
 = \|F^q\|_{L^{\overline{q}(\cdot)}} 
 \lesssim \int_{\R^n} F^q h\,dx \\
 = \int_{\R^n} F^q \prod_{i=1}^m h^{\tk}\,dx 
 \lesssim \int_{\R^n} F^q \prod_{i=1}^m \bigg[ \Rdf_i\big(
  h^{\frac{\overline{q}'(\cdot)q_i}{\overline{p}_i'(\cdot) p_i}}\big)^{\frac{p_i}{q_i}}\bigg]^{\frac{q}{p_i}}
  \,dx. 
\end{multline}
By construction, we have that for each $i$
\[ \Rdf_i\big(
  h^{\frac{\overline{q}'(\cdot)q_i}{\overline{p}_i'(\cdot) p_i}}\big)^{\frac{p_i}{q_i}}\in A_1 \cap
  RH_{q_i/p_i}.  \]
Assume for the moment that the last term in the above inequality is finite.  If it is, then we can apply our hypothesis~\eqref{eqn:xtpl1} and the
generalized H\"older's inequality in the scale of variable Lebesgue
spaces \cite[Theorem.~2.26]{cruz-fiorenza-book} to get
\begin{multline*}
 \int_{\R^n} F^q \prod_{i=1}^m \bigg[ \Rdf_i\big(
  h^{\frac{\overline{q}'(\cdot)q_i}{\overline{p}_i'(\cdot) p_i}}\big)^{\frac{p_i}{q_i}}\bigg]^{\frac{q}{p_i}}
  \,dx  \\
 \lesssim \prod_{i=1}^m \left(\int_{\R^n} f_i^{p_i} 
\Rdf_i\big(
  h^{\frac{\overline{q}'(\cdot)q_i}{\overline{p}_i'(\cdot) p_i}}\big)^{\frac{p_i}{q_i}}
  \,dx
\right)^{\frac{q}{p_i}} 
 \lesssim \prod_{i=1}^m \|f_i^{p_i}\|_{L^{\overline{p}_i(\cdot)}}^{\frac{q}{p_i}} \;
\| \Rdf_i\big(
  h^{\frac{\overline{q}'(\cdot)q_i}{\overline{p}_i'(\cdot) p_i}}\big)^{\frac{p_i}{q_i}} \|_{L^{\overline{p}_i'(\cdot)}}^{\frac{q}{p_i}} .
\end{multline*}
Again by rescaling we have that
$\|f_i^{p_i}\|_{L^{\overline{p}_i(\cdot)}}^{\frac{q}{p_i}}
=\|f_i\|_{L^{p_i(\cdot)}}^q$.  Thus to complete the proof of
inequality~\eqref{eqn:xtpl2} we will show that 
\begin{equation} \label{eqn:rdf-norm}
\| \Rdf_i\big(
  h^{\frac{\overline{q}'(\cdot)q_i}{\overline{p}_i'(\cdot) p_i}}\big)^{\frac{p_i}{q_i}} \|_{L^{\overline{p}_i'(\cdot)}}
=
  \| \Rdf_i\big(
  h^{\frac{\overline{q}'(\cdot)q_i}{\overline{p}_i'(\cdot)
      p_i}}\big) \|_{L^{\sigma_i(\cdot)}}^{\frac{p_i}{q_i}}
\lesssim 1. 
\end{equation}
By the properties of the iteration operator and rescaling,
\[ \| \Rdf_i\big(
  h^{\frac{\overline{q}'(\cdot)q_i}{\overline{p}_i'(\cdot)
      p_i}}\big) \|_{L^{\sigma_i(\cdot)}}^{\frac{p_i}{q_i}}
  \lesssim 
\|  h^{\frac{\overline{q}'(\cdot)q_i}{\overline{p}_i'(\cdot)
    p_i}} \|_{L^{\sigma_i(\cdot)}}^{\frac{p_i}{q_i}}
= \|  h^{\frac{\overline{q}'(\cdot)}{\overline{p}_i'(\cdot)
    }} \|_{L^{\overline{p}_i'(\cdot)}}
. \]
By the relationship between the norm and the modular in variable
exponent spaces \cite[Prop.~2.21]{cruz-fiorenza-book}, since
$\|h\|_{L^{\overline{q}'(\cdot)}} =1$,
\[ 1 = \int_{\R^n} h(x)^{\overline{q}'(x)}\,dx 
= \int_{\R^n} \bigg(
h(x)^{\frac{\overline{q}'(x)}{\overline{p}_i'(x)}}\bigg)^{\overline{p}_i'(x)}\,dx, \]
and this in turn implies that 
\[ \|  h^{\frac{\overline{q}'(\cdot)}{\overline{p}_i'(\cdot)
    }} \|_{L^{\overline{p}_i'(\cdot)}}= 1. \]

\medskip

Finally, to complete the proof we need to justify our assumption that the last term
in~\eqref{eqn:xtrpl3} is finite.   
If we divide the second and last
terms of the identity~\eqref{eqn:tk} by $\overline{q}'(\cdot)$, we get
\[ 1 = \frac{1}{\overline{q}(\cdot)} + \sum_{i=1}^m
  \frac{q}{p_i}\frac{1}{\overline{p}_i'(\cdot)}
= \frac{1}{\overline{q}(\cdot)}.\]
Hence, by the multi-term generalized H\"older's inequality in variable
exponent Lebesgue spaces~\cite[Cor.~2.30]{cruz-fiorenza-book},
\[ \int_{\R^n} F^q \prod_{i=1}^m \bigg[ \Rdf_i\big(
  h^{\frac{\overline{q}'(\cdot)q_i}{\overline{p}_i'(\cdot) p_i}}\big)^{\frac{p_i}{q_i}}\bigg]^{\frac{q}{p_i}}
  \,dx  
\lesssim
\|F^q\|_{L^{\overline{q}(\cdot)}} \prod_{i=1}^m 
\|\bigg[ \Rdf_i\big(
  h^{\frac{\overline{q}'(\cdot)q_i}{\overline{p}_i'(\cdot)
      p_i}}\big)^{\frac{p_i}{q_i}}\bigg]^{\frac{q}{p_i}}\|_{L^{p_i\overline{p}_i'(\cdot)/q}}. \]
By rescaling, the first term becomes $\|F\|_{L^{\qq}}^q$ which is
finite, and by rescaling and ~\eqref{eqn:rdf-norm} we see that the remaining terms are
all uniformly bounded.
\end{proof}

\medskip

Theorem~\ref{th-002} follows directly from Theorem~\ref{thm:xtpl} and
a careful density argument.

\begin{proof}[Proof of Theorem \ref{th-002}]
  From condition \eqref{eqn:r-defn} we can find $p_i>0$ such that
  $p_i<[p_i(\cdot)]_-$ and
  \[ \frac{1}{p_1}+\cdots+\frac{1}{p_m} > \frac{\gamma}{n}. \]
  Therefore, by Theorem~\ref{thm:genver}, given $q_i>p_i$ such
  that $w_i \in RH_{q_i/p_i}$, inequality~\eqref{eq-006} holds.  We
  can use this to apply Theorem~\ref{th-002} if we can define the
  appropriate family $\F$ of extrapolation $(m+1)$-tuples.

 Since each $p_i(\cdot)$ is log-H\"older continuous and
 $[p_i(\cdot)]_->0$, $1\leq i \leq m$, there exists an $N$ depending
 only on the $p_i(\cdot)$ and on $n$ such that functions of the form
\begin{equation*} %\label{eqn:sum}
 f = \sum_{j=1}^M \lambda_j a_j,
\end{equation*}
where each $a_j$ is an $(N,\infty)$ atom, are dense in
$H^{p_i(\cdot)}$~\cite[Theorem.~6.3]{DCU-dw-P2014}.  All such functions
are also contained in $H^p(w)$, for any $p>0$ and $w\in A_\infty$.  
Denote the set of such functions by
$\mathcal{A}$.    Define the family of $(m+1)$-tuples
$\F=\{(f_1,\ldots,f_m,F_R)\}$, where $f_i=M_{N_0}g_i$, $g_i \in
\mathcal{A}$, $R>0$, and
\[ F_R = \min\big( |T_\gamma(g_1,\ldots,g_m)|,R\big)\chi_{B(0,R)}.  \]
Since
$\|\chi_{B(0,R)}\|_{L^\qq}<\infty$~\cite[Lemma~2.39]{cruz-fiorenza-book},
 we have that $\|F_R\|_{L^\qq}<\infty$.  
Further, given any weights $w_i \in RH_{q_i/p_i}$, and
$\overline{w}=\prod_{i=1}^m w_i^{q/p_i}$, by H\"older's inequality
with exponents $q_i/q$ we have that 
\[ \|F_R\|_{L^q(\overline{w})} \leq
  R\left(\int_{B(0,R)} \prod_{i=1}^m w_i^{q/p_i}\,dx\right)^{1/q}
  \leq R \prod_{i=1}^m \left(\int_{B(0,R)} w_i^{q_i/p_i}\,dx\right)^{1/q_i}
< \infty. \]
But then by~\eqref{eq-006} we have that  given any
 $(m+1)$-tuple in $\F$,
\[ \|F_R\|_{L^p(\overline{w})} \lesssim \prod_{i=1}^m
  \|g_i\|_{H^{p_i}(w_i)}
= \prod_{i=1}^m \|f_i\|_{L^{p_i}(w_i)}, \]
which gives us~\eqref{eqn:xtpl1}.  Therefore, by
Theorem~\ref{thm:xtpl}, 
\[ \|F_R\|_{L^\qq} \lesssim \prod_{i=1}^m \|f_i\|_{L^{p_i(\cdot)}} =
  \prod_{i=1}^m \|g_i\|_{H^{p_i(\cdot)}}. \]
By Fatou's lemma in the  variable exponent Lebesgue
spaces~\cite[Theorem.~2.61]{cruz-fiorenza-book},
\[ \|T(g_1,\ldots,g_m)\|_{L^\qq} \leq \liminf_{R\rightarrow \infty}
\|F_R\|_{L^\qq}
 \lesssim \prod_{i=1}^m \|g_i\|_{H^\qk}. \]
This establishes the desired norm inequality of $T$ for a dense family of
functions, and Theorem~\ref{th-002} follows by a standard approximation argument.
\end{proof}

\bibliographystyle{plain}

%\nocite{*}
%%Build the local bib file%%%
%%no .bib extension%%%%%
%\bibliography{FIOHardyBIB}
\end{document}